\documentclass[11pt,twoside, reqno]{amsart}
\usepackage{latexsym}
\usepackage{amsmath}
\usepackage{amssymb}
\usepackage{exercise}
\usepackage{dsfont}
\usepackage{wasysym}

\usepackage[bookmarks=true,hyperindex,pdftex,colorlinks,citecolor=blue]{hyperref}

\newcommand{\eps}{\varepsilon}

\newcommand{\N}{{\mathbb N}}

\newcommand{\F}{\mathfrak{F}}

\newcommand{\flim}{\lim\limits_\F }

\newcommand{\Id}{\mathrm{Id}}

\newcommand{\ADM}{\mathrm{ADM}}

\theoremstyle{plain}
\newtheorem{theorem}{Theorem}[section]
\newtheorem{lemma}[theorem]{Lemma}
\newtheorem{corollary}[theorem]{Corollary}
\newtheorem{proposition}[theorem]{Proposition}
\newtheorem{example}[theorem]{Example}

\theoremstyle{remark}
\newtheorem{remark}[theorem]{Remark}
\theoremstyle{definition}
\newtheorem{definition}[theorem]{Definition}
\numberwithin{equation}{section}

\begin{document}
\title[Norms of partial sums for an $\F$-basis]{Norms of partial sums operators for a basis with respect to a filter}
\author[Kadets]{V. Kadets}

\address[Kadets]{ \href{http://orcid.org/0000-0002-5606-2679}{ORCID: \texttt{0000-0002-5606-2679}} {School of Mathematical Sciences, Holon Institute of Technology (Israel)}}
\email{kadetsv@hit.ac.il}

\author[Manskova]{M. Manskova}

\address[Manskova]{ \href{http://orcid.org/0009-0005-4012-4790}{ORCID: \texttt{0009-0005-4012-4790}} {Graz University of Technology (Austria)
}}
\email{maryna.manskova@tugraz.at}

\subjclass[]{46B20; 46B15; 54A20}
\keywords{Banach space; filter; statistical convergence; filter bases in Banach spaces; generalised bases}

\begin{abstract}
Basis of a Banach space with respect to a filter $\F$ on $\N$ ($\F$-basis for short)  is a generalization of basis, where the ordinary convergence of series is substituted by convergence of partial sums with respect to the filter $\F$. We study the behavior of the norms of partial sums operators for an $\F$-basis, depending on the filter and on the space. One of the central results is:

  The following properties of a sequence $(a_n)_{n \in \N} \subset (1, \infty)$ are equivalent:
\begin{enumerate}
\item[(i)] $\sum_{n \in \N} a_n^{-1} = \infty$.

\item[(ii)] There are a free filter $\F$ on $\N$, an infinite-dimensional Banach space $X$ and an $\F$-basis  $(u_k)$ of $X$ such that the norms of the partial sums operators with respect of $(u_k)$ are equal to the corresponding $a_n$.
\end{enumerate}

\end{abstract}

\thanks{The project was initiated in 2021-2022 as a M.Sc. research project of the second author under the supervision of the first author when both of the authors were affiliated with V.N.Karazin Kharkiv National University. The project was interrupted because of the Russian invasion to Ukraine and resumed three years later.}

\maketitle

%%%%%%%%%%%%%%%%%%%%%%%%%%%%%%%%%%%%%%%%%%%%%%%%%%%%%%%%%%%%%%%%%%%%%%%%%%%%%%

\section{Introduction}

Below, the letters $X,Y,E$ are reserved for real infinite-dimensional Banach spaces, $X^*$ stands for the dual Banach space to $X$, $L(X,Y)$ denotes the Banach space of all continuous linear operators $T: X \to Y$, and $L(X) :=L(X,X)$. We use the standard terminology and notation from Functional Analysis like, for example, in \cite{kadbook}. In particular \cite[Section 16.1]{kadbook} contains some basic information about filters and filter convergence, see also the introductory part of Section \ref{sec_3} below.

A sequence $(e_n)_{n \in \mathbb N}$ in $X$ is said to be a Schauder basis (or just a basis) if for every $x \in X$ there is a unique sequence of scalars $(a_n)_{n \in \mathbb N}$  such that $ \sum_{k=1}^\infty a_k e_k = x$. Denote by $e_n^* (x)$ the coefficients of the decomposition of $x$ in the basis $\{{e_n}\}_1^\infty$, and by $S_n (x)$ the $n$-th partial sum of the decomposition, i.e., $S_n (x) = \sum\limits_{k=1}^n {e_k^* (x)e_k}$. It is easy to see that  $e_n^*$ are linear functionals on $X$ (they are called \emph{coordinate functionals}), and $S_n$ are linear operators (called \emph{partial sums operators}) acting from $X$ into $X$.

A highly non-trivial classical result due to S.~Banach says that the coordinate functionals and the partial sum operators are continuous (i.e. $e_n^*\in X^*$ and $S_n \in L(X)$) and, moreover, 
\begin{equation} \label{intr-eq1}
\sup_n \|{S_n}\| = C < \infty, 
\end{equation}
see, for example, \cite[Section 10.5.2]{kadbook}.  

Basis with respect to a filter is a generalization of basis, where the ordinary convergence of series is substituted by convergence of partial sums with respect to a filter.

\begin{definition}[\cite{gakad}] \label{defF-bas}
 Let $\mathfrak F$ be a free filter on $\mathbb N$. A sequence $(e_n)_{n \in \mathbb N}$ in a Banach space $X$ is said to be an $\mathfrak F$-basis if for every $x \in X$ there is a unique sequence of scalars $(a_n)_{n \in \mathbb N}$  such that $\mathfrak F\)-\(\lim_n \sum_{k=1}^n a_k e_k = x$.
\end{definition}

Like in the particular case of Schauder basis, in the case of $\mathfrak F$-basis the corresponding coordinate functionals $e_n^*: x \mapsto a_n$ and the partial sums operators $S_n: x \mapsto \sum_{k=1}^n a_k e_k$ are linear.
It is an open question, asked explicitly by the first author in 2011 \cite{Kad-prob}, whether the coordinate functionals $x \mapsto a_n$ are necessarily continuous. In fact the question arose already in \cite{gakad}, where it was asked in the following equivalent form ``Is it true that every $\mathfrak F$-basis ought to be a minimal system?''.

The problem was addressed in several papers \cite{kochanek2011, kochanek2012, kania}, which recently leaded to the positive answer for analytic filters \cite[Theorem A]{RKS23}.  So, although for general filters the problem remains to be unsolved, for all ``good'' filters that are defined by a kind of explicit formula, the answer is known to be positive. Moreover, \cite[Theorem B]{RKS23} says that if an $\F$-basis  $(e_n)_{n \in \N}$ is a minimal system than there is an analytic filter $\tilde\F$ such that $(e_n)_{n \in \N}$ is an $\tilde\F$-basis.

As we already remarked, the Banach's theorem not only states the continuity of $S_n$ for a Schauder basis, but also the uniform boundedness of the sequence $(S_n)$. The latter result is not true for $\mathfrak F$-bases in general because by the pointwise convergence criterion \cite[Theorem 2 of Section 10.4.2]{kadbook} under the condition \eqref{intr-eq1} the $\mathfrak F$-basis becomes automatically a Schauder basis (for an $\mathfrak F$-basis $(e_n)_{n \in \mathbb N} \subset X$ the sequence $(S_n)$ converges pointwise to the identity operator $\Id$ on the linear span of $(e_n)_{n \in \mathbb N}$ which is dense in $X$;  together with \eqref{intr-eq1} this gives the pointwise convergence to $\Id$ on the whole space $X$).

The paper is devoted to the following natural question: given a free filter  $\mathfrak F$, what restrictions on the norms of partial sums of an $\mathfrak F$-basis one gets? Of course, when we address this question, we assume that $S_n$ are continuous, otherwise we cannot speak about their norms.

The structure of the paper is as follows. In Section \ref{sec2} we analyze those sequences $(a_n)$ of positive reals that can serve as norms of functionals   $x_n^* \in X^*$ such that the zero element of $X^*$ is a cluster point of the sequence $(x_n^*)$ in the $w^*$ topology $\sigma(X^*, X)$. This class of sequences $(a_n)$ depends on the space $X$. After that, in Section \ref{sec_3}, we turn to  those sequences $(a_n)$ of positive reals that can serve as norms of functionals $x_n^* \in X^*$  whose $w^*$-limit with respect to a given filter $\F$ is equal to zero. This class of sequences $(a_n)$ depends both on the space $X$ and on the filter $\F$. After this preparatory work is done, in Section \ref{sec_4} we present the main results on the norms of partial sums of an $\mathfrak F$-basis. The last section ``The role of summable filters'' addresses a special class of filters which happens to be important for our considerations.

%%%%%%%%%%%%%%%%%%%%%%%%%%%%%%%%%%%%%%%%%%%%%%%%%%%%%%%%%%%%%%%%%%%%%%%%%%%%%%

\section{$(w^*, X^*)$-acceptable sequences} \label{sec2}

Let $X$ be an infinite-dimensional Banach space. Following \cite{kaleor} we call a sequence $(a_n)_{n \in \N}$ of positive reals \emph{$X$-acceptable} if there is a sequence $(x_n)_{n \in \N} \subset X $, $\|x_n\|=a_n$ for which zero is a weak cluster point. In \cite{Kad-cyl} (see Theorem \ref{theo-Kad-cyl1} below) it was proved that if $X$ is a Hilbert space then $X$-acceptability of $\bar a = (a_n)_{n \in \N}$ is equivalent to $\sum_{n \in \N} a_n^{-2}=\infty$. For $X = c_0$
(or more generally for spaces where $c_0$ is finitely representable) $X$-acceptability of $\bar a$ is equivalent to $\sum_{n \in \N}
a_{n}^{-1}=\infty$. So $X$-acceptability really depends on $X$.

For the current research a similar notion of $(w^*, X^*)$-acceptability is of crucial importance.

\begin{definition} \label{def-w*accept}
A sequence $(a_n)_{n \in \N}$ of positive reals is said to be \emph{$(w^*, X^*)$-acceptable} if there is a sequence $(x_n^*)_{n \in \N} \subset X^* $, $\|x_n^*\|=a_n$ for which zero element of $X^*$ is a cluster point in the $w^*$ topology $\sigma(X^*, X)$.
\end{definition}

In the case of reflexive space $E$, where $(E^*)^* = E$, $(w^*, E^*)$-acceptability is equivalent to $X$-acceptability for $X = E^*$, which enables us for reflexive spaces to use in the setting of $(w^*, E^*)$-acceptability the results from \cite{Kad-cyl} demonstrated (without using this notation) for $X$-acceptability. There is one evident connection more that follows from the fact that the topology $\sigma(E^*, E)$ is weaker than $\sigma(E^*, E^{**})$:

\begin{proposition} \label{propX-vers_W*}
Let $E$ be a Banach space and $(a_n)_{n \in \N}$ be an $E^*$-acceptable sequence. Then $(a_n)_{n \in \N}$ is $(w^*, E^*)$-acceptable.
\end{proposition}

So, let us list the results about $(w^*, E^*)$-acceptability that follow directly from \cite{Kad-cyl}.

\begin{theorem}[{\cite[Theorem 3.1 and Proposition 2.2]{Kad-cyl}}]
\label{theo-Kad-cyl1}  

Let $H$ be an infinite-dimensional separable Hilbert space (say, $H=\ell_2$), and $(e_n) \subset H$ be an orthonormal basis of $H$.The following properties of a sequence $(a_n) \subset {\mathbb{R}}^+$ are equivalent:

\begin{enumerate}
\item[(i)] $(a_n)$ is $(w^*, H^*)$-acceptable;

\item[(ii)] The sequence $(a_n e_n)$ has 0 as a weak cluster point.

\item[(iii)] $\sum_1^\infty a_n^{-2}=\infty$.
\end{enumerate}
\end{theorem}

The crucial ingredient of the proof of the above result was the K.~Ball's Complex plank problem theorem \cite{comp}, see its modern proof and generalizations in \cite{GKP24}.

\begin{theorem}[{\cite[Corollary 5.2]{Kad-cyl}}]
\label{theo-Kad-cyl2}  
Every sequence $(a_n) \subset {\mathbb{R}}^+$ satisfying $\sum_1^\infty a_n^{-2}=\infty$ is $X^*$-acceptable and hence $(w^*, X^*)$-acceptable for every $X$.
\end{theorem}

The following general result does not follow formally from \cite{Kad-cyl} although it was briefly sketched there for the case of $X$-acceptability.

\begin{theorem}
\label{theo-Kad-Ball}  
If for some $X$ the sequence $(a_n)_{n \in \N} \subset {\mathbb{R}}^+$  is $(w^*, X^*)$-acceptable than it satisfies the condition 
\begin{equation*} \label{sec2-eq1}
\sum_1^\infty a_n^{-1}=\infty. 
\end{equation*}
\end{theorem}
\begin{proof}
Assume to the contrary that
\begin{equation} \label{sec2-eq2}
\sum_1^\infty a_n^{-1}=R<\infty. 
\end{equation}
By Definition \ref{def-w*accept}, there is a sequence $(x_n^*)_{n \in \N} \subset X^* $, $\|x_n\|=a_n$, for which zero element of $X^*$ is a cluster point in the $w^*$ topology $\sigma(X^*, X)$. Denote
$$
P_n=\left\{x \in X : \left|x_n^*(x)\right| \le 1\right\}=\left\{x \in X : \left|\frac{x_n^*}{\|x_n^*\|}(x)\right| \le a_n^{-1}\right\}.
$$
In the terminology of \cite{ball1}, $P_n$ are planks of half-widths $a_n^{-1}$. Then \cite[Theorem 1]{ball1} together with \eqref{sec2-eq2} imply that the planks $P_n$ cannot cover the whole space $X$: they even cannot cover
a ball of radius $R + \eps$. So there is an element $x \in X$ for which all the inequalities $\left|x_n^*(x)\right| > 1$, $n = 1,2, \ldots$ hold true at the same time. This $x$ separates our sequence $(x_n^*)$ from 0, so 0  is not a cluster point of $(x_n^*)$ in the $w^*$ topology.
\end{proof}

Outside of Hilbert spaces, where Theorem \ref{theo-Kad-cyl1} gives a complete description of $(w^*, E^*)$-acceptability, we have another complete description for $E^*=\ell_\infty$ and generally, for the spaces $E^*$ where  $\ell_\infty$ is finitely representable, that is spaces that are not C-convex. See a short introduction to finite representability and the theory of  C-convex spaces in \cite[Chapter 5]{kadkad}. For this description we need the following known result:

\begin{theorem}[{\cite[Corollary 5.3]{Kad-cyl}}] \label{theor-ell_inf}
Let $X$ be a Banach space in which $\ell_\infty$ is finitely
representable. Then the following properties for a sequence of $a_n>0$ are equivalent:
\begin{enumerate}
\item There is a sequence of $x_n \in X$ with $\|x_n\|=a_n$, having 0 as a weak cluster point;
\item $\sum_1^\infty a_n^{-1}=\infty$.
\end{enumerate}
\end{theorem}

Also, we need a useful lemma.

\begin{lemma}  \label{lem1}
Let $X$ be a Banach space and a sequence $(g_n) \subset X^*\setminus \{0\}$ has the following property $(A)$:  there is a constant $C>0$ such that for every $x \in X$
$$
\sum_{m=1}^\infty \left|g_m(x) \right| \le C \|x\|.
$$
Then, for every sequence of positive reals $(a_n)$ satisfying $\sum_1^\infty a_n^{-1}=\infty$ the corresponding sequence  $(a_n g_n) \subset X^*$ has 0 as a $w^*$-cluster point.
\end{lemma}
\begin{proof}

According to the definition of the topology $\sigma(X^*, X)$, we need to demonstrate that for every finite collection of  $x_k \in X$, $k=1, 2,\ldots,s$ there is an $m \in \N$ such that

\begin{equation} \label{sec2-eq3+}
\max_{1\le k \le s} \left| a_m g_m(x_k)  \right| <1. 
\end{equation}

In order to show \eqref{sec2-eq3+} it is sufficient to demonstrate that   
there is an $m \in \N$ such that

\begin{equation*} \label{sec2-eq4+}
\sum_{ k = 1}^s \left|  a_m g_m(x_k) \right| <1. 
\end{equation*}

Assume to the contrary that there is a finite collection of  $x_k \in X$, $k=1, 2,\ldots,s$ such that for all $m \in \N$

\begin{equation} \label{sec2-eq5+}
\sum_{ k = 1}^s \left|  a_m g_m(x_k) \right| \ge 1. 
\end{equation}

Introduce the following weights:
$$
p_{n,m}:=\frac{a_m^{-1}}{\sum_{j=1}^n a_j^{-1}}, m=1,2, \ldots, n; \sum_{m=1}^n p_{n,m} =1.
$$
Then \eqref{sec2-eq5+} implies that for every $n \in \N$
$$
\sum_{m=1}^n p_{n,m} \sum_{ k = 1}^s \left|  a_m g_m(x_k) \right| \ge 1. 
$$
On the other hand,
$$
\sum_{m=1}^n p_{n,m} \sum_{ k = 1}^s \left|  a_m g_m(x_k) \right| =  \sum_{ k = 1}^s\sum_{m=1}^n p_{n,m} \left|  a_m g_m(x_k) \right|  
$$
$$
=  \sum_{ k = 1}^s\frac{\sum_{m=1}^n
\left| g_m(x_k) \right|}{\sum_{j=1}^n a_j^{-1}} \le  \sum_{ k = 1}^s\frac{C\|x_k\|}{\sum_{j=1}^n a_j^{-1}} \xrightarrow[ n \to \infty]{} 0,
$$
which is a contradiction.
\end{proof}

\begin{remark} \label{remarkrem}
From the Uniform Boundedness Principle one can easily deduce that the condition (A) from  Lemma \ref{lem1} is equivalent to  \emph{$w^*$-absolute convergence} of the series $\sum_{m=1}^\infty g_m$, that is to condition that for every $x \in X$
$$
\sum_{m=1}^\infty \left|g_m(x) \right|< \infty,
$$
see \cite[Lemma 6.4.1]{kadkad} for a completely analogous statement. In reality, $w^*$-absolute convergence of the series $\sum_{m=1}^\infty g_m$ is equivalent to its weak absolute convergence, which can be demonstrated through the reformulation given in \cite[Exercise 6.4.1]{kadkad}, but all these subtleties stay too far from the subject of our paper.
\end{remark}

Now we are ready for the promiced description.
\begin{theorem}  \label{theorE*=ell_inf}
Let $E$ be a Banach space such that $\ell_\infty$ is finitely
representable in $E^*$. Denote $(e_n) \subset c_0 \subset \ell_\infty$ the canonical basis of $c_0$.Then the following properties of a sequence $(a_n) \subset {\mathbb{R}}^+$ are equivalent:

\begin{enumerate}
\item[(i)] $(a_n)$ is $(w^*, E^*)$-acceptable;

\item[(ii)] The sequence $(a_n e_n)$ has 0 as a $\sigma(\ell_\infty, \ell_1)$-cluster point.

\item[(iii)] $\sum_1^\infty a_n^{-1}=\infty$.
\end{enumerate}
\end{theorem}

\begin{proof}
The implications (i)$\Longrightarrow$(iii) and (ii)$\Longrightarrow$(iii) follow from Theorem \ref{theo-Kad-Ball}. The implication (iii)$\Longrightarrow$(i) is a part of Theorem \ref{theor-ell_inf} combined with Proposition \ref{propX-vers_W*}. Finally, the implication (iii)$\Longrightarrow$(ii) is hidden in several different parts of \cite{Kad-cyl} which makes impossible a direct reference. By this reason, instead of a reference we are going to apply Lemma \ref{lem1} to the sequence  $(e_n) \subset \ell_1^* = \ell_\infty$. For this, it is sufficient to demonstrate the validity of condition (A) from that Lemma with $C=1$. Indeed, take an arbitrary $x=(x_1, x_2, \ldots) \in \ell_1$. Then
$$
\sum_{m=1}^\infty \left|e_m(x) \right| = \sum_{m=1}^\infty \left|x_m\right| = \|x\|.
$$
\end{proof}

Theorems \ref{theo-Kad-cyl1} and \ref{theorE*=ell_inf} give us complete descriptions of $(w^*, E^*)$-acceptability for $E = \ell_2$ and  $E = \ell_1$. What is known for other spaces $\ell_p$? The part that concerns the behavior of sequences of the form  $(a_n e_n)$, where  $(e_n)$ is the corresponding canonical basis, generalizes in a natural way.

\begin{theorem}[{\cite[Lemma 2.2]{Shkarin}}]  \label{theorE*=ell_p}
Let $p, p' \in (1, \infty)$, $\frac{1}{p}+\frac{1}{p'}=1$ and $(e_n)$ be the  canonical basis of $\ell_{p'}$.Then the following properties of a sequence $(a_n) \subset {\mathbb{R}}^+$ are equivalent:

\begin{enumerate}

\item[(A)] The sequence $(a_n e_n)$ has 0 as a $\sigma(\ell_{p'}, \ell_p)$-cluster point.

\item[(B)] $\sum_1^\infty a_n^{-p}=\infty$.
\end{enumerate}
\end{theorem}

This says to us that any $(a_n) \subset {\mathbb{R}}^+$ that satisfies (B) is $(w^*,\ell_p^*)$-acceptable. The inverse implication that $(w^*,\ell_p^*)$-acceptability implies the above condition (B) is incorrect for the case of $p > 2$ by the following reason. For $p > 2$ the condition (B) is stronger than the condition $\sum_1^\infty a_n^{-2}=\infty$. But on the other hand, every sequence with $\sum_1^\infty a_n^{-2}=\infty$ is $(w^*,\ell_p^*)$-acceptable due to Theorem \ref{theo-Kad-cyl2}. So, it remains to verify whether $(w^*,\ell_p^*)$-acceptability implies (B) for $1<p<2$. To the best of our knowledge, this question is open. There is a strong result from \cite{Shkarin} that ``almost'' resolves this question in positive, see also \cite{bay} for generalization to other spaces.

\begin{theorem}[{\cite[Proposition 5.3]{Shkarin}}]  \label{theorE*=ell_p++}
Let $1<p<2$ and the sequence $(a_n) \subset {\mathbb{R}}^+$ be $(w^*,\ell_p^*)$-acceptable. Then, for every $1<s<p$, 
$$
\sum_1^\infty a_n^{-s}=\infty.
$$
\end{theorem}

%%%%%%%%%%%%%%%%%%%%%%%%%%%%%%%%%%%%%%%%%%%%%%%%%%%%%%%%%%%%%%%%%%%%%%%%%%%%%%

\section{ $(w^*, E^*, \F)$-acceptable sequences} \label{sec_3}

In this section we already need some notation and basic results about filters. In order to make the reading more comfortable, we repeat below with minor modifications the corresponding introductory part from \cite[Section 3]{kaleor}. 

A \textit{filter} $\F$ on a set $N$ is a non-empty
collection of subsets of $N$ satisfying the following axioms: $\emptyset
\notin \F$; if $A,B \in \F$ then $A \cap B \in \F$; and for every $A \in \F$ if $B \supset A$ then $B \in \F$. All over the paper we consider filters on ${\mathbb{N}}$.

The dual to the notion of filter is the notion of \textit{ideal}. An ideal $%
\mathcal{I}$ on ${\mathbb{N}}$ is a family of subsets of ${\mathbb{N}}$
closed under taking finite unions and subsets of its elements. Given a
filter $\F$ on ${\mathbb{N}}$ we have the corresponding ideal of
complements $\mathcal{I}_{\F}=\{{\mathbb{N}}\setminus A:\ A\in%
\F\}$ on ${\mathbb{N}}$. And vice versa the filter $\F_{%
\mathcal{I}}=\{{\mathbb{N}}\setminus A:\ A\in\mathcal{I}\}$ corresponds to a
given ideal $\mathcal{I}$. The elements of $\mathcal{I}_{\F}$ are
called $\F$-\textit{negligible}. Sometimes it is more convenient to
present a filter by pointing out its ideal.

Let $X$ be a topological space (in our paper it will usually be a Banach space equipped with one of the standard in Functional Analysis topologies). A sequence $(x_n) \subset X$, $n \in  \N$,  is said to be $\F$-\textit{convergent} to $x$ if for every neighborhood $U$
of $x$ the set $\{n \in \N: x_n \in U\}$ belongs to $\F$
(equivalently $\{n \in \N: x_n \notin U\} \in \mathcal{I}_{%
\F}$). We write this as $x=\flim x_n$, or $x_n \to_\F x$, or, if the variable should be pointed explicitly, $x=\F
\text{-}\lim_n x_n$.

 In particular if one takes as $\F$ the filter whose
ideal consists of finite sets (the \textit{Fr\'echet filter}), then $%
\F$-convergence coincides with the ordinary one.

The natural ordering on the set of filters on $\N$ is defined as
follows: $\F_1 \succ \F_2$ if $\F_1 \supset
\F_2$.
% If $G$ is a centered collection of subsets (i.e. all finite
%intersections of the elements of $G$ are non-empty), then there is a filter
%containing all the elements of $G$. The smallest filter, containing all the
%elements of $G$ is called \textit{the filter generated by} $G$.
A filter $\F$ on $\N$ is said to be \textit{free} if it
dominates the Fr\'echet filter. Below we deal only with free filters. In this case every ordinary convergent sequence is automatically $\F$-convergent.

A subset of $\N$ is called \textit{stationary} with respect to $%
\F$ (or just $\F$-stationary) if it has nonempty
intersection with each member of the filter. In other words, an $A \subset {%
\mathbb{N}}$ is $\F$-stationary if and only if it does not belong
to $\mathcal{I}_{\F}$. Denote the collection of all $\F$%
-stationary sets by $\F^*$. 
For an $I \in \F^*$ we call
the collection of sets $\{A\cap I:\ A\in\F\}$ \textit{the trace of $%
\F$ on $I$} (which is evidently a filter on $I$), and by $\F(I)$ we denote the filter on $\N$ generated by the trace of $\F$ on $I$. Clearly $\F(I)$ dominates $\F$. 
Any subset of $\N$ is either a member of $\F$, or a member of $\mathcal{I}_{\F}$, or the set and its complement are both $\F$-stationary sets.

\begin{theorem}[{\cite[Theorem 1.1]{shur}}] \label{stationary-thm} 
Let $X$ be topological space, $x_n, x \in
X$ and let $\F$ be a filter on $\N$. Then the following
conditions are equivalent

\begin{enumerate}
\item[(i)] $(x_n)$ is $\F$-convergent to $x$;

\item[(ii)] $(x_n)$ is $\F(I)$-convergent to $x$ for every $I \in
\F^*$;

\item[(iii)] $x$ is a cluster point of $(x_n)_{n \in I}$ for every $I \in
\F^*$.
\end{enumerate}
\end{theorem}
\begin{proof} The proof is taken from \cite[Theorem 1.1]{shur} with misprints corrected.
Implications (i)$\Longrightarrow$(ii)
and (ii)$\Longrightarrow$(iii) are evident. Let us
prove that (iii)$\Longrightarrow$(i). Suppose $x_n$
do not $\F$-converge to $x$. Then there is  such a neighborhood
$U$ of $x$  that in each $A\in\F$ there is a $j \in A$ such that
$x_j \not\in U$. Consequently $I=\{j \in \N: x_j \not\in U\}$ is
stationary and at the same time $x$ is not a cluster point of $(x_n)_{n \in I}$. This contradicts the assumption (iii).
\end{proof}
More about filters, ultrafilters and their applications one can find in advanced General Topology textbooks, for example in \cite{tod}.

\vspace{1 mm} 

Now, let us start the main part of the section. 

Let $E$ be an infinite-dimensional Banach space, and $\F$ be a free filter on $\N$.
\begin{definition} \label{def-w*-F-accept}
 A sequence $(a_n)_{n \in \N}$ of positive reals is said to be \emph{$(w^*, E^*, \F)$-acceptable} if there is a sequence $(x_n^*)_{n \in \N} \subset E^*$, $\|x_n^*\|=a_n$, for which $\flim  x_n^*(x)=0$ for each $x \in E$ (or, in other words, $\flim x_n^* = 0$ in the $w^*$ topology $\sigma(E^*, E)$). Denote $\mathcal A(w^*, E^*, \F)$ the collection of all $(w^*, E^*, \F)$-acceptable sequences. 
\end{definition}

Our first result on $(w^*, E^*, \F)$-acceptability follows from Theorem \ref{theo-Kad-Ball}.

\begin{theorem} \label{thmA(gener)}
For every  infinite-dimensional Banach space $E$ and every free filter $\F$ on $\N$ every  $(w^*, E^*, \F)$-acceptable sequence $(a_n)_{n \in \N}$ satisfies the following condition:

\centerline{ $\sum\limits_{n \in I} a_n^{-1} = \infty$ for every $I \in \F^*$.}
\end{theorem}

\begin{proof}
According to the defininition,  there is a sequence $(x_n^*)_{n \in \N} \subset E^*$ with $\|x_n^*\|=a_n$ for which  $\flim x_n^*= 0$ in the  topology $\sigma(E^*, E)$. Then (Theorem \ref{stationary-thm}), for every $I \in \F^*$, zero is a $\sigma(E^*, E)$-cluster point of  $(x_n)_{n \in I}$ for every $I \in \F^*$. This means that, enumerating $I$ as $I=\{n_1, n_2, \dots\}$, we get an  $(w^*, E^*)$-acceptable sequence $( \|x_{n_k}^*\|)_{k \in \N}=(a_{n_k})_{k \in \N}$. To conclude the proof it remains to apply  Theorem \ref{theo-Kad-Ball} to the sequence $(a_{n_k})_{k \in \N}$:  $\sum\limits_{n \in I} a_n^{-1} = \sum\limits_{k \in \N} a_{n_k}^{-1} = \infty$.
\end{proof}

It happens that at least for one space $E$ the condition from the previous theorem is not only a necessary condition of  $(w^*, E^*, \F)$-acceptability, but also a sufficient one.

\begin{theorem}\label{thmA(ellinf)}
Let  $(e_n)$ be the canonical basis of $c_0$, $\F$ be a free filter on $\N$, and  $(a_n)_{n \in \N}$ be a sequence of positive reals. Then the following conditions are equivalent

\begin{enumerate}
\item[(i)]  $(a_n)_{n \in \N} \in \mathcal A(w^*, (\ell_1)^*, \F)$;

\item[(ii)] $\flim a_n e_n = 0$ in the $w^*$-topology of $\ell_\infty = (\ell_1)^*$;

\item[(iii)]  for every $I \in \F^*$ zero is a $\sigma(\ell_\infty, \ell_1)$-cluster point of  $(a_n e_n)_{n \in I} \subset \ell_\infty$;

\item[(iv)]  for every $I \in \F^*$
$$
\sum\limits_{n \in I} a_n^{-1} = \infty.
$$
\end{enumerate}
\end{theorem}
\begin{proof}
The implication (i)$\Longrightarrow$(iv) is a particular case of Theorem \ref{thmA(gener)}, the implication (iv)$\Longrightarrow$(iii) is a part of Theorem \ref{theorE*=ell_inf}, (iii)$\Longrightarrow$(ii) because of Theorem \ref{stationary-thm}, and the last implication (ii)$\Longrightarrow$(i) is a consequence of Definition \ref{def-w*-F-accept}. 
\end{proof}

\begin{remark} \label{remthmA(ellinf)}
By the same reason as above, the equivalence   (i)$\Longleftrightarrow$(iv) works for every space $E$ for which $E^*$ contains an isomorphic copy of $c_0$. It remains unclear for us whether this equivalence remains valid for all spaces $E$ for which $\ell_\infty$ is finitely representable in $E^*$.
\end{remark}

Even though for some spaces the necessary condition from  Theorem \ref{thmA(gener)} is also a sufficient one, this is not always the case. For example, this does not work for $E = \ell_2$. Moreover, 
for $\ell_2$, and for every  infinite-dimensional separable Hilbert space as well, the dual space is canonically identified with the original one and the $w^*$-topology is the same as the weak topology, so the case of $E=\ell_2$ is covered by the corresponding results from \cite{kaleor} that was formulated originally in terms of the weak topology.

\begin{theorem}[{\cite[Theorem 3.4]{kaleor}}] \label{thmA(H)}
Let $H$ be an infinite-dimensional separable Hilbert space, $(e_n) \subset H$ be an orthonormal basis of $H$, $\F$ be a free filter on $\N$, and  $(a_n)_{n \in \N}$ be a sequence of positive reals. Then the following conditions are equivalent

\begin{enumerate}
\item[(i)]  $(a_n)_{n \in \N} \in \mathcal A(w^*, H^*, \F)$;

\item[(ii)] for every $h \in H$
$$
\flim \langle a_n e_n, h \rangle = 0;
$$

\item[(iii)]  for every $I \in \F^*$
$$
\sum\limits_{n \in I} a_n^{-2} = \infty.
$$
\end{enumerate}
\end{theorem}

Because of this, the description of $\mathcal A(w^*, E^*, \F)$ for each concrete space $E$ is a separate problem that (at least for us) looks interesting. Remark that this problem remains open even for $E=\ell_p$, $p \in (1, 2) \cup (2, \infty)$. From theorems \ref{theorE*=ell_p} and \ref{theorE*=ell_p++}, following the lines of the proof of Theorem \ref{thmA(ellinf)}, one can get necessary conditions and sufficient conditions that do not coincide but are relatively close one to the other. 

\begin{theorem}  \label{theorE*=ell_pforF}
Let $1<p<2$,  $\F$ be a free filter on $\N$ and $(a_n)_{n \in \N}$ be a sequence of positive reals. Then 
\begin{enumerate}
\item[--]  if  for every $I \in \F^*$
$$
\sum\limits_{n \in I} a_n^{-p} = \infty
$$
then  $(a_n)_{n \in \N} \in \mathcal A(w^*, \ell_p^*, \F)$;

\item[--]  if $(a_n)_{n \in \N} \in \mathcal A(w^*, \ell_p^*, \F)$ then for every $1<s<p$ and  every $I \in \F^*$
$$
\sum\limits_{n \in I} a_n^{-s} = \infty.
$$
\end{enumerate}
\end{theorem}

\begin{definition} \label{def-p-admis}
Let $1 \le p \le 2$ and  $\F$ be a free filter on $\N$. A sequence $(a_n)_{n \in \N}$ of positive reals is said to be \emph{$(\F, p)$-admissible} if for every $I \in \F^*$
$$
\sum\limits_{n \in I} a_n^{-p} = \infty.
$$
Denote $\ADM_p(\F)$ the collection of all $(\F, p)$-admissible sequences.
\end{definition}

In this terminology, Theorems \ref{thmA(H)}, \ref{thmA(ellinf)} and \ref{theorE*=ell_pforF} can be summarized as follows: $\mathcal A(w^*, \ell_2^*, \F) = \ADM_2(\F)$,  $\mathcal A(w^*, \ell_1^*, \F) = \ADM_1(\F)$ and  $\mathcal A(w^*, \ell_p^*, \F)$ for $1<p<2$ is ``close'' to $\ADM_p(\F)$.

Remark that, for a given filter $\F$ on $\N$, one
needs a substantial amount of work to decide whether a sequence  $(a_n)_{n \in \N}$ is or is not $(\F, p)$-admissible. For $(\F, 2)$-admissibility, this work has been performed in \cite[Sections 4 and 5]{kaleor} for several important filters, including the filter of statistical convergence $\F_{st}$. For example, \cite [Corollary 4.3]{kaleor} says that a non-decreasing sequence $(a_n) \subset {\mathbb{R}}^+$ is $(\F_{st}, 2)$-admissible if and only if
\begin{equation}  \label{a-mon}
\sup_n \frac{a_{n}}{n^{1/2}} < \infty.
\end{equation}
Analogous statements for $p \neq 2$ were not considered in \cite{kaleor} but they can be deduced the same way.  

%%%%%%%%%%%%%%%%%%%%%%%%%%%%%%%%%%%%%%%%%%%%%%%%%%%%%%%%%%%%%%%%%%%%%%%%%%%%%%

\section{Partial sum operators with respect to an $\F$-basis} \label{sec_4}

In the previous sections we introduced the necessary tools and now we are ready for the main subject. As we noticed in the Introduction, we are going to figure out how fast the norms of partial sum operators with respect to an $\F$-basis can tend to infinity, depending on the Banach space $X$ and on the free filter $\F$ on $\N$. We already used many times the notation $(e_n)$ for the canonical basis of $\ell_p$ or $c_0$. In order to avoid possible confusion, in this section we switch to other letters like $u_n$ or $v_n$ for an $\F$-basis. We study only those $\F$-bases $(u_n)$ for which the corresponding coordinate functionals $(u_n^*)$ are continuous. The action of  the corresponding partial sums operators $S_n$ can be expressed by the formula
$$
S_n(x)=\sum_{k=1}^n u_k^*(x)u_n.
$$
The operators $S_n$ are projections on their images $\textrm{span}\{u_k\}_{k=1}^n$, consequently $\|S_n\| \ge 1$. The reminders $R_n = \Id - S_n$ are projections as well, so $\|R_n\| \ge 1$. By the triangle inequality, $\left|\|R_n\| - \|S_n\|\right| \le \|\Id\| = 1$. Together, this implies that 
\begin{equation} \label{ineqSnRn}
\|R_n\| \le \|S_n\| +1 \le 2\|S_n\| \textrm{ and } \|S_n\| \le \|R_n\| +1 \le 2\|R_n\|.
\end{equation}

\begin{theorem}  \label{f-op-thm1} Let $X, Y$ be Banach spaces, $\bar a = (a_n)_{n \in {\mathbb{N}}}$ be a sequence of positive reals and $\F$ be a free filter on $\N$.
Then the following two statements are equivalent:

\begin{itemize}
\item[(1)] There is a sequence of functionals $x_n^* \in X^*$ such that $\|x_n^*\|=a_n$, $n \in {\mathbb{N}}$, and for every $x \in X$
$$
\flim x_n^*(x) = 0.
$$

\item[(2)] There is a sequence of operators $T_n \in L(X,Y)$, $n \in {\mathbb{N}}$, such that $\|T_n\|=a_n$, $n \in {\mathbb{N}}$, and for every $x \in X$
$$
\flim T_n(x) = 0.
$$
\end{itemize}
\end{theorem}

\begin{proof}
The implication (1)$\Longrightarrow$(2) is evident: it is sufficient to define $T_n$ by the rule $T_n(x) = x_n^*(x)y_0$, where $y_0 \in Y$ is an arbitrary fixed element of $\|y_0\| = 1$.

Now, let us demonstrate the implication (2)$\Longrightarrow$(1). Assume that $T_n$, $n \in {\mathbb{N}}$, are the operators from the item (2). Consider $T_n^* \in L(Y^*,X^*)$. Since  $\|T_n^*\| = \|T_n\|=a_n$, for each $n \in {\mathbb{N}}$ there is $y_n^* \in S_{Y^*}$ such that $\|T_n^*(y_n^*)\|  > \frac{a_n}{2}$. Let us denote 
$$
x_n^*=\frac{a_n}{\|T_n^*(y_n^*)\|}T_n^*(y_n^*)
$$
and check that these functionals $x_n^* \in X^*$ are what we need.

Evidently, $\|x_n^*\|=a_n$. Next,  for every $x \in X$
$$
\flim \left|x_n^*(x)\right| = \flim \frac{a_n}{\|T_n^*(y_n^*)\|}\left|(T_n^*(y_n^*))(x)\right|
$$
$$
= \flim \frac{a_n}{\|T_n^*(y_n^*)\|}\left|y_n^*(T_n x)\right|\leq \flim \frac{a_n}{\|T_n^*(y_n^*)\|}\left\|y_n^*\right\|\left\|T_n x\right\|
$$
$$
\leq 2\flim \left\|T_n x\right\| = 0.
$$
\end{proof}

\begin{corollary}  \label{cor-rem1}
Definition \ref{defF-bas} of $\F$-basis means that for every $x \in X$
$$
\flim R_n(x) = 0.
$$
So, the previous theorem implies that the norms of reminders $\|R_n\|$, $n \in \N$, form an $(w^*, X^*, \mathfrak F)$-acceptable sequence.
\end{corollary}

The next theorem lists applications of the above result together with results from Section \ref{sec_3}.

\begin{theorem}  \label{thm-F-bas-nec}
Let $\F$ be a free filter on $\N$. Then
\begin{enumerate}
\item[(i)] for every infinite-dimensional Banach space $X$ and every $\F$-basis $(u_n)$ of $X$ with continuous coordinate functional the sequence $\left(\|S_n\|\right)$ of norms of the corresponding partial sums operators is $(\F, 1)$-admissible.

\item[(ii)] If $X=\ell_2$ then for every $\F$-basis $(u_n)$ of $X$ with continuous coordinate functional the sequence $\left(\|S_n\|\right)$ is $(\F, 2)$-admissible.

\item[(iii)]If $X=\ell_p$ with $1 < p < 2$ then for every $\F$-basis $(u_n)$ of $X$ with continuous coordinate functional the sequence $\left(\|S_n\|\right)$ is $(\F, s)$-admissible for all $s \in [1, p)$.
\end{enumerate}
\end{theorem}

\begin{proof}
According to Corollary \ref{cor-rem1},  in all the statements (i), (ii), (iii) the norms of reminders $\|R_n\|$, $n \in \N$, form an $(w^*, X^*, \mathfrak F)$-acceptable sequence. So, applying Theorems \ref{thmA(gener)}, \ref{thmA(H)} and \ref{theorE*=ell_pforF}, we see that all the statements (i), (ii), (iii) would be correct with  $\left(\|R_n\|\right)$ instead of $\left(\|S_n\|\right)$. In order to get what we need it is sufficient to apply the inequalities \eqref{ineqSnRn}.
\end{proof}

The previous theorem gives necessary conditions on a sequence $(a_n)$ of positive numbers to be norms of partial sums operators with respect to an $\F$-basis. In order to get sufficient conditions we are going to use the following construction which generalizes the construction from \cite[Theorem 1]{Kad-bib-ip}.

\begin{example} \label{example1}
Let $U=(u_k)_{k \in \N}$ be an $\F$-basis of a Banach space $X$, $(u^*_k)_{k \in \N}$ be its coordinate functionals, and $S_n$ be the corresponding partial sum operators. For a given sequence of positive scalars $B=(b_n)_{n \in \N}$ we define a new system of vectors $V(U, B) = (v_k)_{k \in \N}$ as follows:
$$
v_n=\sum\limits_{i=1}^n b_i u_i.
$$
Denote 
$$
v_n^*=\frac{1}{b_n}u^*_n-\frac{1}{b_{n+1}}u^*_{n+1}.
$$
Note that $\{v_n,v_n^*\}_{n=1}^{\infty}$ is a biorthogonal system. Indeed, 
\begin{align*}
    v_m^*(v_n)&=\left(\frac{1}{b_m}u^*_m-\frac{1}{b_{m+1}}u^*_{m+1}\right)\left( \sum_{i=1}^n b_i u_i\right)\\ 
    &=\sum_{i=1}^n\frac{b_i}{b_m}u_m^*(u_i)-\sum_{i=1}^n\frac{b_i}{b_{m+1}}u_{m+1}^*(u_i)=\delta_{m,n}.
\end{align*}
Denote 
$$
\tilde S_n(x)=\sum\limits_{k=1}^n v_k^*(x)v_k
$$
the partial sum projections associated with this new biorthogonal system.
For each $x \in X$,   
\begin{align} \label{eq-ex-newsums}
\tilde S_n(x)&=\sum\limits_{k=1}^n v_k^*(x)v_k =\sum\limits_{i=1}^n u_i^*(x) u_i - \frac{u_{n+1}^*(x)}{b_{n+1}}\sum\limits_{i=1}^n b_iu_i \nonumber \\ &= S_n(x) - \frac{u_{n+1}^*(x)}{b_{n+1}}\sum\limits_{i=1}^n b_iu_i.
\end{align}
Since $\flim S_n(x) = x$ for all $x \in X$, the system $V(U, B)$ forms an $\F$-basis if and only if for all $x \in X$
$$
\left\|\frac{u_{n+1}^*(x)}{b_{n+1}}\sum\limits_{i=1}^n b_iu_i\right\| = \frac{1}{b_{n+1}} \left(\left\|\sum\limits_{i=1}^n b_i u_i\right\| u_{n+1}^*\right)(x)  \to_\mathfrak{F} 0.
$$
In other words, the system $V(U, B)$ forms an $\F$-basis if and only if the sequence of functionals
$$
\frac{1}{b_{n+1}} \left\|\sum\limits_{i=1}^n b_i u_i\right\| u_{n+1}^* \in X^*
$$
is $\F$-convergent to zero in the $w^*$-topology.
\end{example}

\begin{theorem}  \label{thm-F-bas-1}
Let $\F$ be a free filter on $\N$ and $ (a_n)_{n \in \N} \subset (1, \infty)$ be a sequence of reals. Then the following conditions are equivalent:
\begin{enumerate}
\item[(i)] There is an $\F$-basis  of $\ell_1$ such that the norms of the partial sums operators with respect of that $\F$-basis are equal to the corresponding $a_n$.

\item[(ii)] There is an infinite-dimensional Banach space $X$ and an $\F$-basis of $X$ such that the norms of the partial sums operators with respect of that $\F$-basis are equal to the corresponding $a_n$.

\item[(iii)] The sequence $ (a_n)_{n \in \N}$ is $(\F, 1)$-admissible.
\end{enumerate}
\end{theorem}
\begin{proof}
The implication (i)$\Longrightarrow$(ii) is evident and (ii)$\Longrightarrow$(iii) follows from (i) of Theorem \ref{thm-F-bas-nec}. Let us demonstrate  that (iii) implies (i). Take $X=\ell_1$ and let $ (e_n)_{n \in \N}$ be the canonical basis of $\ell_1$. Then the corresponding coordinate functionals $(e_n^*)_{n \in \N}$ form the canonical basis of $c_0$. Denote $S_n$ be the partial sum operators with respect to the basis $ (e_n)_{n \in \N}$.  We are going to apply the construction from Example \ref{example1} with $(u_k)_{k \in \N} = (e_n)_{n \in \N}$. For this,  we define the corresponding positive scalars $b_n$ recurrently. Put $b_1=1$ and, when $b_1, \ldots , b_n$  are already defined, we select $b_{n+1}$ in such a way that
\begin{equation} \label{eq-ex-newsums+}
\left\| S_n - \frac{e_{n+1}^*}{b_{n+1}}\otimes \sum\limits_{i=1}^n b_ie_i\right\| = a_n.
\end{equation}
(here, for a functional $f \in X^*$ and for an element $z \in X$ we use the notation $f\otimes z$ for the operator that acts from $X$ to $X$ by the rule $(f\otimes z)(x) = f(x) z$).
Such a selection is possible because the function 
$$
f(t):= \left\|  S_n - \frac{e_{n+1}^*}{t}\otimes \sum\limits_{i=1}^n b_ie_i\right\|
$$
is continuous on $(0, +\infty)$, $\lim_{t \to +0}f(t) = +\infty$ and $\lim_{t \to +\infty}f(t) = 1$.
Then 
$$
v_n=\sum\limits_{i=1}^n b_i e_i, v_n^*=\frac{1}{b_n}e^*_n-\frac{1}{b_{n+1}}e^*_{n+1}.
$$
from Example \ref{example1} form a biorthogonal system with the corresponding partial sums operators being
$$
\tilde S_n(x)= S_n(x) - \frac{e_{n+1}^*(x)}{b_{n+1}}\sum\limits_{i=1}^n b_ie_i,
$$
and our choice of $(b_n)$ ensures that $ \left\|\tilde S_n\right\| = a_n$ as required. So, it remains to check that $(v_n)_{n \in \N}$ form an $\F$-basis of $X$. 
Denote 
$$
c_n=\frac{1}{b_{n+1}} \left\|\sum\limits_{i=1}^n b_i e_i\right\|. 
$$
As remarked at the end of Example \ref{example1}, $(v_n)_{n \in \N}$ form an $\F$-basis of $X$ if and only if the sequence $(c_n e_n^*)$ is $\F$-convergent to zero in the $w^*$-topology of $\ell_\infty = (\ell_1)^*$. 

 The recurrent condition \eqref{eq-ex-newsums+} and the triangle inequality imply that
\begin{equation} \label{ineqcnan}
|c_n - a_n|=\left|\left\|\frac{e_{n+1}^*}{b_{n+1}}\otimes \sum\limits_{i=1}^n b_ie_i\right\| - \left\| S_n - \frac{e_{n+1}^*}{b_{n+1}}\otimes \sum\limits_{i=1}^n b_ie_i\right\|\right| \le \|S_n\| = 1.
\end{equation}
The condition (ii) of our Theorem says that the sequence $ (a_n)_{n \in \N}$ is $(\F, 1)$-admissible, so by Theorem \ref{thmA(ellinf)} $\flim a_n e_n^* = 0$ in the $w^*$-topology. Together with \eqref{ineqcnan} this means that $\flim c_n e_n^* = 0$ in the $w^*$-topology as well.
\end{proof}

There is a similar characterization in the case of Hilbert spaces.

\begin{theorem}  \label{thm-F-bas-2}
Let $\F$ be a free filter on $\N$ and $ (a_n)_{n \in \N} \subset (1, \infty)$ be a sequence of reals. Then the following conditions are equivalent:
\begin{enumerate}
\item[(i)] There is an $\F$-basis of an infinite-dimensional Hilbert space $H$ such that the norms of the partial sums operators with respect of that $\F$-basis are equal to the corresponding $a_n$.

\item[(ii)] The sequence $ (a_n)_{n \in \N}$ is $(\F, 2)$-admissible.
\end{enumerate}
\end{theorem}
\begin{proof}
The implication (i)$\Longrightarrow$(ii) follows from (ii) of Theorem \ref{thm-F-bas-nec}. The inverse implication (ii)$\Longrightarrow$(i) is analogous to the corresponding part of the previous theorem. Take $H=\ell_2$ and  apply the construction from Example \ref{example1} with $(u_k)_{k \in \N} = (e_n)_{n \in \N}$ where $ (e_n)_{n \in \N}$ is the canonical basis of $\ell_2$ and the same recurrent construction of  $b_n$ by the rule from \eqref{eq-ex-newsums+} with  $b_1=1$.
All the proof goes the same way as before, just at the very end, instead of Theorem \ref{thmA(ellinf)} one needs to apply Theorem \ref{thmA(H)}.
\end{proof}

In the same vein, one gets the following partial result about $\F$-bases in $\ell_p$.

\begin{theorem}  \label{thm-F-bas-3}
Let $\F$ be a free filter on $\N$, $1 < p < 2$, and $ (a_n)_{n \in \N} \subset (1, \infty)$ be an $(\F, p)$-admissible sequence of reals. Then there is an $\F$-basis of $\ell_p$ such that the norms of the partial sums operators with respect of that $\F$-basis are equal to the corresponding $a_n$.
\end{theorem}

%%%%%%%%%%%%%%%%%%%%%%%%%%%%%%%%%%%%%%%%%%%%%%%%%%%%%%%%%%%%%%%%%%%%%%%%%%%%%%

\section{The role of summable filters} \label{sec_5}

One can ``measure'' filters comparing them with some standard explicitly defined filters. In this section we use as ``scale of measurement'' the summable filters that are defined below.

\begin{definition} \label{defsum-fil}
For a sequence $s=(s_k)$ of non-negative real numbers such that $\sum_{k=1}^{\infty} s_k=\infty$ the corresponding \textit{summable filter} $\F^{s}$ is the collection of those subsets $A\subset{\mathbb{N}}$ that $\sum_{k\in \N \setminus A} s_k < \infty$. 
\end{definition}

Remark, that for the filter  $\F = \F^{s}$ the corresponding ideal $\mathcal{I}^{s}:=\mathcal{I}_{\F^{s}}$ consists of complements to elements of $\F^{s}$, that is of those sets $B \subset \N$ that $\sum_{k\in B} s_k < \infty$. Consequently, $D \in \left(\F^{s}\right)^*$ if and only if $\sum_{k\in D} s_k = \infty$.

\begin{proposition}[{Adaptation of \cite[Corollary 5.16]{kaleor}}] \label{prop-sum-acc}
Let $s=(s_k) \subset [0, 1]$ be a sequence reals such that  $\sum_{k=1}^{\infty} s_k=\infty$, and let  $a=(a_k)$ be a sequence of positive numbers. Then the following conditions are equivalent:
\begin{enumerate}
\item[(i)] the sequence $a$ is $(\F^{s}, 1)$-admissible.

\item[(ii)] $(a_n s_n)$ is $\F^{s}$-bounded, i.e., there is an $A \in \F^{s}$ such that
$\sup_{n \in A}a_n s_n = M < \infty$.
\end{enumerate}
\end{proposition}
\begin{proof}
(ii)$\Longrightarrow$(i). Consider an arbitrary $I \in \left(\F^{s}\right)^*$ and let $A \in \F^{s}$ and $M$ be from condition (ii).  Then $\sum\limits_{n \in I} s_n = \infty$, $\sum\limits_{n \in \N \setminus A} s_n < \infty$, and consequently
$$
 \sum\limits_{n \in I} a_n^{-1} \ge  \sum\limits_{n \in I \cap A} a_n^{-1} \ge \frac1M  \sum\limits_{n \in I \cap A} s_n \ge \frac1M  \sum\limits_{n \in I} s_n - \frac1M  \sum\limits_{n \in \N \setminus A} s_n =  \infty.
$$
(i)$\Longrightarrow$(ii). Assume to the contrary that $\sup_{n \in A}a_n s_n = \infty$ for every $A \in \F^{s}$. Then, for every $M > 0$, the set $B_M = \{n \in \N: a_n s_n > M \}$ intersects all elements $A \in \F^{s}$. This means that $B_M \in (\F^{s})^*$, that is 
$$
\sum_{n \in B_M} s_n = \infty.
$$
Using the last condition recurrently for $M = 2, 4, 8, \ldots$ we can select disjoint finite subsets $D_1, D_2, \ldots$ such that $D_m \subset B_{2^m}$ and
$$
1 \le \sum_{n \in D_m} s_n \le 2.
$$
Consider the set $D = \bigsqcup_{m \in \N} D_m$. By our construction, 
$$
\sum_{n \in D} s_n = \sum_{m \in \N} \sum_{n \in D_m} s_n = \infty
$$
which means that $D \in (\F^{s})^*$.

On the other hand,
$$
\sum_{n \in D} a_n^{-1} = \sum_{m \in \N} \sum_{n \in D_m} a_n^{-1} \le  \sum_{m \in \N} \sum_{n \in D_m}2^{-m} s_n \le \sum_{m \in \N} 2^{-m} \cdot 2 < \infty,
$$
which contradicts the condition (i).
\end{proof}

\begin{proposition} \label{prop-sum-fil++}
Let $\F$ be a free filter on $\N$. Then a sequence  $(a_n)$ of positive numbers with $\sum_{n \in \N} a_n^{-1} = \infty$  is $(\F, 1)$-admissible if and only if $\F$ dominates the summable filter $\F^{s}$ for  $s=\left(a_n^{-1}\right)$. Consequently, if for $\F$ there is and $\F$-basis of some Banach space $X$ then $\F$ dominates the summable filter $\F^{s}$ for $s=\left(\|S_n\|^{-1}\right)$.
\end{proposition}
\begin{proof}
Assume that  $(a_n)$ is $(\F, 1)$-admissible. Then for every $I \in \mathfrak{F}^*$
$$
 \sum\limits_{n \in I} a_n^{-1} = \infty.
$$
Consequently, every $I \in \mathfrak{F}^*$ belongs to $\left(\F^{s}\right)^*$, that is $\mathfrak{F}^* \subset \left(\F^{s}\right)^*$. This means that  $\F \succ \F^{s}$.

Conversely, assume that  $\F \succ \F^{s}$. Then every $I \in \mathfrak{F}^*$ belongs to $\left(\F^{s}\right)^*$ as well, that is $\sum\limits_{n \in I} a_n^{-1} = \infty$.
\end{proof}

\begin{theorem}  \label{thm-sum-fil1}
Let $(a_n)_{n \in \N} \subset (1, \infty)$.  Then the following conditions are equivalent:
\begin{enumerate}
\item[(i)] $\sum_{n \in \N} a_n^{-1} = \infty$.

\item[(ii)] There are a free filter $\F$ on $\N$, an infinite-dimensional Banach space $X$ and an $\F$-basis  $(u_k)_{k \in \N}$ of $X$ such that the norms of the partial sums operators with respect of $(u_k)_{k \in \N}$ are equal to the corresponding $a_n$.

\item[(iii)] For $s = \left(a_n^{-1}\right)_{n \in \N}$, there is an $\F^s$-basis of $\ell_1$ such that the norms of the partial sums operators with respect of $(u_k)_{k \in \N}$ are equal to the corresponding $a_n$.
\end{enumerate}
\end{theorem}
\begin{proof}
The implication (iii)$\Longrightarrow$(ii) is evident. Now, assuming (ii), we obtain the $(\F, 1)$-admissibility of  $(a_n)$ (see (i) of Theorem \ref{thm-F-bas-nec}). Since $\N \in \F^*$, this gives us our condition (i).   So, the implication  (ii)$\Longrightarrow$(i) is demonstrated. The remaining implication (i)$\Longrightarrow$(iii) follows from Theorem \ref{thm-F-bas-1} because $(a_n)$ is $(\F^s, 1)$-admissible.
\end{proof}

In our opinion, the most intriguing open problem about $\F$-bases is whether for every separable Banach space $X$ there is a free filter  $\F$ on $\N$ such that $X$ possesses an $\F$-basis. If one fixes $\F$, one gets a related problem whether for $X$ possesses an $\F$-basis for given $\F$. Historically, the corresponding problem for $\F$ being the Fr\'echet filter (that is the problem whether every separable Banach space $X$ possesses a Schauder basis) was very stimulating for the Banach space theory for decades and was solved by Enflo in 1972 \cite{enflo}. To the best of our knowledge, the problem remains open for a number of concrete filters, including the filter of statistical convergence. The next theorem provides the negative answer for a wide class of filters.

\begin{definition} \label{defslowfil}
A a free analytic filter $\mathfrak F$ on $\mathbb N$ is said to be \emph{slow} if for every $(\F, 1)$-admissible sequence of scalars $(a_n) \subset [1, \infty)$ 
$$
\inf_n \frac{a_n}{\sqrt{n}} = 0.
$$
\end{definition}

Remark, that non-trivial examples of slow filters one easily gets from Proposition \ref{prop-sum-acc}. Say, for $(s_n) = (n^{-\alpha})$ with $\alpha \in (0, 1/2)$ the corresponding summable filter $\F^{s}$ is a slow filter.

\begin{theorem}  \label{thm-slowfil}
There is a separable Banach space $X$ that has no $\F$-basis with respect to any    slow filter $\F$. 
\end{theorem}
\begin{proof}
We are going to apply the following result by Pisier \cite[Corollary 10.8]{Pisier1986}: there is  a  separable Banach space $X$ and a $\delta > 0$ such that for every projection $P$ in $X$ with finite-dimensional range
$$
\|P\| \ge \delta \sqrt{\dim P(X)}.
$$
Assuming that $\F$ is a slow filter and  $(u_k)_{k \in \N}$ is an $\F$-basis of this Pisier's space $X$ we easily get a contradiction. Namely, in this case  the norms $\|S_n\|$ of the partial sums operators with respect of $(u_k)_{k \in \N}$ should satisfy the condition 
$$
\|S_n\| \ge \delta \sqrt{n},
$$
but on the other hand, by Theorem \ref{thm-F-bas-nec} the sequence $\left(\|S_n\|\right)$ of norms of the corresponding partial sums operators is $(\F, 1)$-admissible, so due to Definition \ref{defslowfil},
$$
\inf_n \frac{\|S_n\|}{\sqrt{n}} = 0.
$$
\end{proof}

Remark, that one cannot construct an analogue of Pisier's space with the function $\sqrt{\cdot}$ substituted by a function of quicker growth because the classical Kadets-Snobar theorem  \cite[Theorem B]{RKS23} says that for every finite-dimensional subspace $Y$ of a Banach space $E$ there is a projection $P: E \to Y$ with $\|P\| \le \sqrt{\dim Y}$.

\section*{Acknowledgments}

The first author acknowledges support by the KAMEA program administered by the Ministry of Absorption, Israel.
The second author was supported by the Austrian Science Fund (FWF), project DOC-183.

\end{document}